\documentclass{amsart}

\usepackage{amsmath}
\usepackage{amsfonts}
\usepackage{amsthm}
\usepackage{amssymb}
\usepackage{bm}
\usepackage[applemac]{inputenc}
\usepackage{enumerate}

\newtheorem{thm}{Theorem}[section]
\newtheorem{cor}[thm]{Corollary}
\newtheorem{Prop}[thm]{Proposition}
\newtheorem{lemme}[thm]{Lemma}
\newtheorem{fait}[thm]{Fact}
\newtheorem{conj}[thm]{Conjecture}
\newtheorem{qu}[thm]{Question}

\newtheorem{ex}[thm]{Example}
\newtheorem{rem}[thm]{Remark}

\newcommand{\monster}{\mathcal U}
\newcommand{\pred}{\mathbf P}
\newcommand{\pprec}{\prec^{+}}
\newcommand{\ordI}{\mathcal I}
\newcommand{\ordJ}{\mathcal J}

\DeclareMathOperator{\tp}{tp}

\DeclareMathOperator{\VC}{VC}

\def\indsym#1#2{%
 \setbox0=\hbox{$\m@th#1x$}%
 \kern\wd0%
 \hbox to 0pt{\hss$\m@th#1\mid$\hbox to 0pt{$\m@th#1^{#2}$\hss}\hss}%
 \lower.9\ht0\hbox to 0pt{\hss$\m@th#1\smile$\hss}%
 \kern\wd0}

\def\nindsym#1#2{%
 \setbox0=\hbox{$\m@th#1x$}%
 \kern\wd0%
 \hbox to 0pt{\hss$\m@th#1\not$\kern1.4\wd0\hss}
 \hbox to 0pt{\hss$\m@th#1\mid$\hbox to 0pt{$\m@th#1^{#2}$\hss}\hss}%
 \lower.9\ht0\hbox to 0pt{\hss$\m@th#1\smile$\hss}%
 \kern\wd0}

\begin{document}

\title{Invariant types in NIP theories}
\author{Pierre Simon}

\address{Pierre Simon\\
CNRS\\
Institut Camille Jordan\\
Universit\'e Claude Bernard - Lyon 1\\
43 boulevard du 11 novembre 1918\\
69622 Villeurbanne Cedex, France}

\email{simon@math.univ-lyon1.fr}

\urladdr{http://www.normalesup.org/\textasciitilde{}simon/}

\maketitle

\begin{abstract}
We study invariant types in NIP theories. Amongst other things: we prove a definable version of the $(p,q)$-theorem in theories of small or medium directionality; we construct a canonical retraction from the space of $M$-invariant types to that of $M$-finitely satisfiable types; we show some amalgamation results for invariant types and list a number of open questions.
\end{abstract}
%
%\keywords{Model theory; NIP; invariant types; (p,q)-theorem}
%
%\ccode{Mathematics Subject Classification 2000: 03C45}

\section*{Introduction}
An $M$-invariant type is a type over the monster model which is invariant under all automorphisms fixing the model $M$ pointwise. There are two well-known classes of $M$-invariant types: types finitely satisfiable in $M$ and types definable over $M$. In a stable theory, all $M$-invariant types are both definable over $M$ and finitely satisfiable in $M$. This is no longer true in NIP theories. An old theorem of Poizat and Shelah says that a theory is NIP if and only if the number of $M$-invariant types for any model $M$ is at most $2^{|M|}$, if and only if there are at most $2^{|M|}$ global types finitely satisfiable in $M$(\cite{PoiInstable}, \cite{PoiPostScript}). In fact, an $M$-invariant type is determined by the type of its Morley sequence over $M$. This basic observation already suggests that invariant types should have special properties in NIP theories.

Any type finitely satisfiable in $M$ is the limit of some ultrafilter on $M$. We will show that if $M$ and $L$ are countable, then any such type is in fact a limit of a \emph{sequence} of elements of $M$. More generally, we observe that a theory in a countable language is NIP if and only if any sequence $(p_i:i<\omega)$ of types (over any set) has a converging subsequence. This is related to some results in Banach space theory. See Rosenthal's theorem (\cite{Banach_topics}) and also the work of Bourgain, Fremlin and Talagrand \cite{BFT} which uncovers an analogue of the IP/NIP dichotomy in the context of measurable functions on Polish spaces.

An important result about NIP families is the $(p,q)$-theorem of finite combinatorics, proved by Alon-Kleitman and Matou\u sek. This theorem states roughly the following: Let $\phi(x;y)$ be an NIP formula and fix $p\geq q$ large enough, then there is $N$ with the following property: If $\mathcal S=\{\phi(x;b_i):i<n\}$ is a finite family of non-empty instances of $\phi(x;y)$ such that out of every subfamily $\mathcal S_0\subseteq \mathcal S$ of size $p$, there is $\mathcal S_1\subseteq \mathcal S_0$ of size $q$ whose conjunction is consistent, then there is an $N$ point set intersecting each $\phi(x;b_i)$. This result has proved to be extremely useful in model theory; it is a major ingredient in the proof of the UDTFS conjecture \cite{ExtDef2} and also in the study of generics in definably amenable groups \cite{tamedyn}. We will see in Section \ref{sec_pq} how this theorem is linked with properties of finitely satisfiable types and give a model theoretic proof of a weaker version of it. Also, we prove in Theorem \ref{th_pqdef} a definable version of the $(p,q)$-theorem, assuming that types over countable models have at most $\aleph_0$ coheirs. This generalises previous results in \cite{invdp} where dp-minimality was assumed.

The biggest mystery concerning invariant types has to do with those invariant types that are neither definable nor finitely satisfiable. Let $p$ be a global $M$-invariant $\phi$-type. In general, whether or not $\phi(x;b)$ is in $p$ depends on the full type $\tp(b/M)$ and not only on its restriction to $\phi$-formulas (or rather $\phi^{opp}$-formulas). In Section \ref{sec_localdescription} we propose a new point of view of invariant $\phi$-types which remedies this. An invariant $\phi$-type is described in a way that only involves the formula $\phi$ and not the rest of the language.

An intuition that was already presented in \cite{invdp} is that amongst $M$-invariant types, $M$-definable types and $M$-finitely satisfiable types should be seen as two opposite extremes, and other invariant types sit somehow in between. In fact, one should be able to analyse a general invariant type into a finitely satisfiable part and a definable `quotient'. This is very vague, but we will give precise conjectures along this line at the end of the paper. In \cite{invdp}, we gave one piece of evidence in favour of this idea: we showed that a dp-minimal invariant type is either finitely satisfiable in a small model or definable. We will in fact give another, more conceptual, proof of this (Theorem \ref{th_dpmindich}). In Section \ref{sec_fm} we show how to define a canonical retraction $F_M$ from the space of $M$-invariant types to that of $M$-finitely satisfiable types. We prove a number of commutativity properties of this map. This is related to another idea: the relation ``$p$ commutes with $q$", where $p$ and $q$ are global invariant types, is very meaningful and expresses in some way that $p$ and $q$ are \emph{far away} from each other (think of DLO). It is shown in \cite{distal} that two types which commute behave with respect to each other as in a stable theory. For example $p$ has a unique non-forking extension to a realisation of $q$ and the limit type of any Morley sequence of $p$ over a realisation of $q$ is equal to the invariant extension of $p$.

This retraction $F_M$ is interesting in its own right. It is used in \cite{ChePilSim} to show that some properties of definably amenable groups are preserved under taking the Shelah expansion. Nevertheless, this map remains rather puzzling. It would be nice to have a better understanding of it, for instance a different construction leading to it. Also, we did not solve the following question, which, if answered positively, would certainly shed some new light on invariant types: are the fibres of $F_M$ of bounded cardinality (say $2^{|T|}$), independent of $M$?

Finally, in Section \ref{sec_amalg}, we prove some amalgamation properties of invariant types and in Section \ref{sec_questions} we state a number of open questions.

\section{Setting and basic facts}

Throughout, $T$ is a complete theory in a language $L$ and $\monster$ is a monster model. We usually do not assume that $T$ is NIP. If $A\subset \monster$ is a set of parameters, we let $L(A)$ denote the set of formulas with parameters in $A$. We do not distinguish between points and tuples. Thus $a,b,c,\ldots$ usually denote tuples of variables and $a\in A$ means $a\in A^{|a|}$.

We write $M \pprec N$ to mean $M\prec N$ and $N$ is $|M|^+$-saturated.

The notation $\phi^0$ means $\neg \phi$ and $\phi^1$ means $\phi$.

\smallskip
Let $\Delta$ be a set of formulas and $A$ a set of parameters. A (possibly finite) sequence $I=(a_i : i\in \ordI)$ is \emph{$\Delta$-indiscernible} over $A$, if for every integer $k$ and two increasing tuples $i_1 <_{\ordI} \cdots <_{\ordI} i_k$ and $j_1 <_{\ordI} \cdots <_{\ordI} j_k$, $b\in A$ and formula $\phi(x_1,\ldots,x_k;y)\in \Delta$, we have $\phi(a_{i_1},\dots,a_{i_k};b)\leftrightarrow \phi(a_{j_1},\dots,a_{j_k};b)$. An indiscernible sequence is an infinite sequence which is $\Delta$-indiscernible for all $\Delta$.

Let $I=(a_i:i\in \ordI)$ be any infinite sequence. The \emph{Ehrenfeucht-Mostowski type} (or \emph{EM-type}) of $I$ over $A$ is the set of $L(A)$-formulas $\phi(x_1,\ldots,x_k)$ such that $\monster \models \phi(a_{i_1},\ldots,a_{i_k})$ for all $i_1 < \cdots < i_k\in \ordI$, $k<\omega$. If $I$ is an indiscernible sequence, then for every $k$, the restriction of the EM-type of $I$ to formulas in $k$ variables is a complete type over $A$. If $I$ is any sequence and $\ordJ$ is any infinite linear order, then using Ramsey's\index{Ramsey} theorem and compactness, we can find an indiscernible sequence $J$ indexed by $\ordJ$ and realising the EM-type of $I$ (see e.g., \cite[Lemma 5.1.3]{TentZieg}).

Sequences $(I_i:i<k)$ are said to be \emph{mutually indiscernible} over $A$ if each $I_i$ is indiscernible over $A\cup I_{\neq i}$.

\smallskip
Let $\Delta=\{\phi_i(x;y_i) : i<\alpha\}$ be a set of formulas with the same first variable $x$. Then a \emph{$\Delta$-type} over $A$ is a maximal consistent set of formulas of the form $\phi_i(x;b)$ for some $i<\alpha$ and $b\in A$. The set of $\Delta$-types over $A$ is denoted $S_{\Delta}(A)$. If $\Delta=\{\phi(x;y)\}$, we write $\phi$ instead of $\Delta$. The $\Delta$-type $\tp_{\Delta}(a/A)$ of $a$ over $A$ is defined in the obvious way.

If $\phi(x;y)$ is a formula, then we let $\phi^{opp}(y;x)$ denote the opposite formula: $\phi^{opp}(y;x)=\phi(x;y)$, but the roles of variables and parameters are reversed. Hence a $\phi$-type over $A$ is a type in variable $x$ and a $\phi^{opp}$-type over $A$ is a type in variable $y$.

\smallskip
We assume that the reader is familiar with basic facts concerning NIP theories as presented, e.g., in \cite[Chapter 2]{NIPbook}, though we will recall all that we need. The main property, which can be taken as a definition, is the following.

\begin{fait}
A formula $\phi(x;y)$ is NIP if and only if for any indiscernible sequence $I=(a_i:i<\omega)$, for some $n<\omega$, the truth value of $\phi(a_i;b)$ is constant for $n\leq i<\omega$.
\end{fait}

In particular for any $A$, the sequence $(\tp(a_i/A):i<\omega)$ converges. We denote the limit type by $\lim(I/A)$.

By compactness, if $\phi(x;y)$ is NIP, then there is some finite $\Delta$ such that the conclusion holds for any $\Delta$-indiscernible sequence.

\subsection{Invariant types}

We present some basic facts about invariant types in NIP theories. The reader may consult \cite[Chapter 2]{NIPbook} for more information and proofs.

\smallskip
By an $M$-\emph{invariant type}, we always mean a global type $p(x)\in S(\monster)$ which is invariant under automorphisms fixing $M$ pointwise. An \emph{invariant type} is a global type which is $M$-invariant for some $M$. There are two well-known subclasses of invariant types: definable types and finitely satisfiable types. A type $p(x)\in S(\monster)$ is \emph{definable} if for every formula $\phi(x;y)\in L$, there is some $d\phi(y)\in L(\monster)$ such that for any $b\in \monster$, $p\vdash \phi(x;b) \iff \monster \models d\phi(b)$. The type $p$ is \emph{$M$-definable}, or \emph{definable over $M$}, if all the formulas $d\phi(y)$ can be taken with parameters in $M$. A type $p(x)\in S(\monster)$ is \emph{finitely satisfiable} in $M$, or \emph{$M$-finitely satisfiable}, if every formula $\phi(x;b)\in p$ has a realisation in $M$. If $p$ is $M$-invariant and finitely satisfiable in some $N$, then it is $M$-finitely satisfiable, and similarly for definable.

The set of $M$-invariant types is a closed subspace of the space of global types. In particular it is compact. The same is true of the space of global $M$-finitely satisfiable types.

Given $p(x)$ and $q(y)$ two $M$-invariant types, one can define the product $p(x)\otimes q(y)$ as the type $r(x,y)\in S(\monster)$ such that for any $N\supseteq M$, $r|_N = \tp(a,b/N)$ where $b\models q|_N$ and $a\models p|_{Nb}$. This defines an $M$-invariant type. The operation $\otimes$ is associative, but not commutative in general.

We will sometimes write $p_x \otimes q_y$ instead of $p(x)\otimes q(y)$ to make the notation more compact.

If $p(x)$ is an $M$-invariant type, we define by induction on $n\in \mathbb N^*$:\index{$p^{(n)},p^{(\omega)}$}
$$p^{(1)}(x_0) = p(x_0)\quad \text{and} \quad p^{(n+1)}(x_0,\ldots,x_n)=p(x_{n}) \otimes p^{(n)}(x_0,\ldots,x_{n-1}).$$
Let also $p^{(\omega)}(x_0,x_1,\ldots) = \bigcup p^{(n)}$. For any $B \supseteq M$, a realisation $(a_i : i<\omega)$ of $p^{(\omega)}|_B$ is called a \emph{Morley sequence} of $p$ over $B$ (indexed by $\omega$). It follows from associativity of $\otimes$ that such a sequence $(a_i:i<\omega)$ is indiscernible over $B$.

\begin{fait}\label{fact_morley}
Assume that $T$ is NIP. Let $p,q\in S_x(\monster)$ be $M$-invariant types. If $p^{(\omega)}|_M = q^{(\omega)}|_M$, then $p=q$.
\end{fait}

In particular, there are at most $2^{|M|+|T|}$ invariant types over a model $M$. This property actually characterises NIP theories.

A local version of Fact \ref{fact_morley} also holds, with the exact same proof.
\begin{fait}\label{fact_morleyloc}
Assume that the formula $\phi(x;y)$ is NIP. Let $p,q\in S_x(\monster)$ be $M$-invariant types and we let $p_\phi,q_\phi$ denote the restrictions of $p$ and $q$ res\-pectively to formulas of the form $\phi(x;b)$, $b\in \monster$. If $p^{(\omega)}|_M=q^{(\omega)}|_M$, then $p_\phi=q_\phi$.
\end{fait}
\begin{proof}
Assume that for example $p\vdash \phi(x;b)$ and $q\vdash \neg \phi(x;b)$ for some $b\in \monster$. Build inductively a sequence $(a_i:i<\omega)$ such that:

$\cdot$ when $i$ is even, $a_i \models p \upharpoonright Mba_{<i}$;

$\cdot$ when $i$ is odd, $a_i\models q\upharpoonright Mba_{<i}$.

Then by hypothesis, the sequence $(a_i:i<\omega)$ is indiscernible (its type over $M$ is $p^{(\omega)}|_M=q^{(\omega)}|_M$) and the formula $\phi(x;b)$ alternates infinitely often on it, contradicting NIP.
\end{proof}

\subsubsection{Dividing and forking}\label{sec_fork}

Let $A\subseteq B$ and let $\pi(x)$ be a partial type over $B$. We say that $\pi(x)$ \emph{divides} over $A$ if there is an $A$-indiscernible sequence $(b_i:i<\omega)$ and a formula $\phi(x;y)$ such that $\pi(x)\vdash \phi(x;b_0)$ and the partial type $\{\phi(x;b_i):i<\omega\}$ is inconsistent. We say that $\pi(x)$ \emph{forks} over $A$ if it implies a finite disjunction of formulas, each of which divides over $A$.

\begin{fait}\label{fact_fork}
Assume that $T$ is NIP and let $M$ be a model. Then a partial type $\pi(x)$ forks over $M$, if and only if it divides over $M$, if and only if it extends to a global $M$-invariant type.
\end{fait}

We recall the notion of strict non-forking from \cite{CherKapl} (which is used in the proof of the above fact). Let $M$ be a model of an NIP theory. A sequence $(b_i)_{i<\omega}$ is \emph{strictly non-forking} over $M$ if for each $i<\omega$, $\tp(b_i/b_{<i}M)$ is strictly non-forking over $M$, which means that it extends to a global type $\tp(b_*/\monster)$ such that both $\tp(b_*/\monster)$ and $\tp(\monster/Mb_*)$ are non-forking over $M$. We will only need to know two facts about strict non-forking sequences:

\smallskip
(Existence) Given $b\in \monster$ and $M\models T$, there is a sequence $b=b_0,b_1,\ldots$ which is strictly non-forking over $M$. We might call such a sequence a \emph{strict Morley sequence} of $\tp(b/M)$.

\smallskip
(Witnessing property) If the formula $\phi(x;b)$ forks over $M$, then for any strictly non-forking sequence $b=b_0,b_1,\ldots$, the type $\{\phi(x;b_i):i<\omega\}$ is inconsistent.

\smallskip
In addition to \cite{CherKapl}, an exposition of those facts can be found in \cite[Chapter 5]{NIPbook}.

\subsection{Commuting types}

The notion of \emph{commuting types} is central in this work, especially in Section \ref{sec_fm}. Two invariant types $p(x)$ and $q(y)$ \emph{commute} if $p(x)\otimes q(y)=q(y)\otimes p(x)$ as global types. By associativity of $\otimes$, if $p$ and $q$ commute, then $p$ commutes with $q^{(n)}$ for $n\leq \omega$.

As usual, we say that two types $p(x),q(y)\in S(N)$ are \emph{weakly orthogonal} if $p(x)\cup q(y)$ defines a complete type in two variables over $N$. If $p$ and $q$ are $M$-invariant types, we say they are \emph{orthogonal} if they are weakly orthogonal as global types. Note that this implies that $p|_N$ and $q|_N$ are weakly orthogonal for any $N$ such that $M\pprec N$.

Of course, if $p$ and $q$ are orthogonal, then they commute. In NIP theories, we can consider commuting as a kind of weak form of orthogonality. This is motivated by the study of distal theories (see \cite{distal}) where in fact two types commute if and only if they are orthogonal (and this can be taken as a definition of distal theories amongst NIP theories). It is also proved in \cite{distal} that two commuting types behave with respect to each other as do types in a stable theory. One instance of this will be recalled in Proposition \ref{prop_finitecofinite} below.

\begin{lemme}\label{lem_commseq}
($T$ is NIP) Let $p,q$ be $M$-invariant types. Assume that $p^{(\omega)}\otimes q^{(\omega)}|_M = q^{(\omega)}\otimes p^{(\omega)}|_M$, then $p$ and $q$ commute.
\end{lemme}
\begin{proof}
Build a sequence $(a_n,b_n:n<\omega)$ such that $(a_n,b_n)\models p(x)\otimes q(y) | \monster a_{<n}b_{<n}$ when $n$ is even and $(a_n,b_n)\models q(y)\otimes p(x)|\monster a_{<n}b_{<n}$ when $n$ is odd. Using the commutativity hypothesis, one shows by induction that this sequence is indiscernible over $M$. Assume that $p$ and $q$ do not commute: say $p(x)\otimes q(y)\vdash \phi(x,y)$ and $q(y)\otimes p(x)\vdash \neg \phi(x;y)$ for some $\phi(x;y)\in L(\monster)$. Then we have $\models \phi(a_n,b_n) \iff n$ is even, which contradicts NIP.
\end{proof}

Recall that in an NIP theory, a global invariant type $p$ is \emph{generically stable} if it is definable and finitely satisifiable. This is equivalent to saying that $p$ commutes with itself: $p(x_0)\otimes p(x_1)=p(x_1)\otimes p(x_0)$. In fact, a generically stable type commutes with all invariant types. We will not explicitly need this notion in this text, but it is useful to have it in mind.

We recall two results from previous papers (the first one is easy and can constitute an exercise; the second one is more involved).

\begin{lemme}[\cite{invdp} Lemma 2.3]\label{lem_def}
An $M$-invariant type $p$ is definable if and only if for every $M$-finitely satisfiable type $q$, $p\otimes q|_M=q\otimes p|_M$.
\end{lemme}

\begin{Prop}[\cite{distal} Corollary 3.24]\label{prop_finitecofinite}
Let $M\pprec N$ and let $p$ and $q$ be two global $M$-invariant types which commute. Construct $I=(a_i:i<\omega)\models p^{(\omega)}|_N$ and $b\models q|_N$, then $\lim(I/Nb)=p|_{Nb}$.
\end{Prop}

\subsection{Directionality}

In \cite{Sh946}, Kaplan and Shelah classify NIP theories depending on the number of global coheirs a type can have.

A theory $T$ is said to be of \emph{small directionality} if, given a model $M$ and $p\in S(M)$, then for any finite set $\Delta$ of formulas the global coheirs of $p$ determine only finitely many $\Delta$-types. In particular, $p$ has at most $2^{|T|}$ global coheirs. The theory $T$ is of \emph{medium directionality} if it is not of small directionality and if the global coheirs of every such $p$ determine at most $|M|$ $\Delta$-types (and thus $p$ has at most $|M|^{|T|}$ coheirs). Finally, $T$ has \emph{large directionality} if it is NIP, but has neither small nor medium directionality.

\section{$(p,q)$-theorems}\label{sec_pq}

Let $X$ be a set (finite or infinite) and $\mathcal S$ a family of subsets of $X$. Such a pair $(X,\mathcal S)$ is called a set system. We say that the family $\mathcal S$ \emph{shatters} a subset $A\subseteq X$ if for every $A'\subseteq A$, there is a set $S$ in $\mathcal S$ such that $S\cap A=A'$. In other words, the family $\mathcal S$ when restricted to $A$ is the full power set of $A$.

The family $\mathcal S$ has VC-dimension at most $n$ (written $\VC(\mathcal S)\leq n$), if there is no $A\subseteq X$ of cardinality $n+1$ such that $\mathcal S$ shatters $A$. We say that $\mathcal S$ is of VC-dimension $n$ if it is of VC-dimension at most $n$ and shatters some subset of size $n$.

If for each $n$ we can find a subset of $X$ of cardinality $n$ shattered by $\mathcal S$, then we say that $\mathcal S$ has infinite VC-dimension (and write $\VC(\mathcal S)=\infty$).

Given a set system $(X,\mathcal S)$, we define the \emph{dual} set system as the set system $(X^*,\mathcal S^*)$, where $X^*=\mathcal S$ and $\mathcal S^*=\{\mathcal S_a : a\in X\}$ with $\mathcal S_a = \{S\in \mathcal S : a\in S\}$. We then define the dual VC-dimension of $\mathcal S$ (written $\VC^*(\mathcal S)$) as the VC-dimension of $\mathcal S^*$.

If $\phi(x;y)$ is a formula, then $VC(\phi)$ denotes the VC-dimension of the set $\{\phi(\monster;b):b\in \monster\}$. Similarly, $VC^*(\phi)$ denotes the dual VC-dimension of $\{\phi(\monster;b):b\in \monster\}$, or equivalently, the VC-dimension of $\{\phi(a;\monster):a\in \monster\}$.

A formula $\phi(x;y)$ has finite VC-dimension if and only if it has finite dual VC-dimension if and only if it is NIP. For more on this, see \cite[Chapter 6]{NIPbook}.

\smallskip
Let $p\geq q$ be two integers. A family $\mathcal S$ of some set $X$ has the $(p,q)$-property if out of every $p$ sets of $\mathcal S$, some $q$ have non-empty intersection.

\begin{fait}[$(p,q)$-theorem, \cite{pq}]\label{fact_pq}
\index{pqth@$(p,q)$-theorem}
Let $p\geq q$ be two integers. Then there is an integer $N$ such that the following holds:

Let $(X,\mathcal S)$ be a \emph{finite} set system where every $S\in \mathcal S$ is non-empty. Assume:

\quad $\cdot$ $\VC^*(\mathcal S)< q$;

\quad $\cdot$ $(X,\mathcal S)$ has the $(p,q)$-property.

Then there is a subset of $X$ of size $N$ which intersects every element of $\mathcal S$.
\end{fait}

The first $(p,q)$-type theorem was proved by Alon and Kleitman \cite{pq_conv} for families of convex subsets of $\mathbb R^n$ (where in the statement $VC^*(\mathcal S)$ is replaced by the dimension $n$). Then Matou\u sek \cite{pq} showed how to adapt the proof to the case of families of finite VC-dimension. The proof in the special case where $p=q$ is also exposed in \cite[Chapter 6]{NIPbook}.

The following lemma and corollary are well-known.

\begin{lemme}\label{lem_low}
Let $\phi(x;y)$ be a formula of dual VC-dimension $q_0$ and $(b_i:i<\omega)$ an indiscernible sequence of $|y|$-tuples. Assume that the partial type $\{\phi(x;b_i):i<\omega\}$ is $(q_0+1)$-consistent, then it is consistent.
\end{lemme}
\begin{proof}
Assume that the partial type $\{\phi(x;b_i):i<\omega\}$ is $(q_0+1)$-consistent, but $q$-inconsistent for some $q\geq q_0+2$. First increase the sequence $(b_i:i<\omega)$ to an indiscernible sequence $(b_i:i<\omega^2)$. Let $B=\{b_i:i<q_0+1\}$. We will show that $B$ is shattered by the dual family $\{\phi(a;\monster):a\in \monster\}$ thus contradicting the definition of the dual VC-dimension.

By hypothesis, there is some $a_*$ satisfying $\bigwedge_{1\leq k\leq q_0+1} \phi(x;b_{\omega k})$. Let $A$ be the set of indices $l\in \omega^2$ for which $a_*\models \neg \phi(x;b_l)$. By the $q$-inconsistency hypothesis, for every $k$, there are infinitely many elements of $A$ in the interval $\omega k\leq x< \omega (k+1)$. Now fix any $B_0\subseteq B$ and let $\eta:q_0+1 \to \omega^2$ be an increasing map such that $\eta(k)=\omega k$ if $b_k\in B_0$ and $\eta(k)\in A$ otherwise. By indiscernibility, the map $\eta$ extends to an automorphism of $\monster$ which we still call $\eta$. Then, for any $b\in B$, we have $\eta^{-1}(a_*)\models \phi(x;b) \iff b\in B_0$.
\end{proof}

\begin{cor}\label{cor_low}
Let $\phi(x;y)$ be an NIP formula and $M$ a model. Then the set $\{b\in \monster: \phi(x;b)$ does not divide over $M\}$ is open over $M$.
\end{cor}
\begin{proof}
Let $q_0$ be the dual VC-dimension of $\phi(x;y)$.
Let $b\in \monster$. If $\phi(x;b)$ does not divide over $M$, then there is no $M$-indiscernible sequence $(b_i:i<\omega)$ with $b_0=b$ which is $(q_0+1)$-inconsistent. By compactness, there is some formula $\psi(y)\in \tp(b/M)$ which ensures this and then $\phi(x;b')$ does not divide over $M$ for any $b'\models \psi(y)$.
\end{proof}

Let $\phi(x;y)$ and $\psi(y)$ be two formulas over some model $M$. Consider the family $\mathcal S_{\phi,\psi}$=\{$\phi(M;b)$: $b\in \psi(M)$\} of subsets of $M$.

\begin{lemme}\label{lem_pqdef}
The following are equivalent:

$\bullet_1$ The family $\mathcal S_{\phi,\psi}$ has the $(p,q)$-property for some $p\geq q> VC^*(\phi)$;

$\bullet_2$ For every $q$, there is $p\geq q$ such that the family $\mathcal S_{\phi,\psi}$ has the $(p,q)$-property;

$\bullet_3$ For any $b\in \psi(\monster)$, the formula $\phi(x;b)$ does not divide over $M$.
\end{lemme}
\begin{proof}
$3\Rightarrow 2$: Assume the third bullet and pick some $q$. Then we cannot find an $M$-indiscernible $(b_i:i<\omega)$ of elements of $\psi(\monster)$ such that $\{\phi(x;b_i):i<\omega\}$ is $q$-inconsistent. Hence by compactness and Ramsey, for some $p$, the family $\{\phi(\monster;b):b\in \psi(\monster)\}$ has the $(p,q)$-property. This is a first order statement so we can replace $\monster$ with any model, in particular $M$. This implies that $\mathcal S_{\phi,\psi}$ has the $(p,q)$-property.

$2\Rightarrow 1$: Trivial.

$1\Rightarrow 3$: Assume that the first bullet holds and let $(p,q)$ be given by it. Let also $q_0$ be the dual VC-dimension of $\phi(x;y)$. If there is some $M$-indiscernible sequence $(b_i:i<\omega)$ in $\psi(\monster)$ such that $\{\phi(x;b_i):i<\omega\}$ is inconsistent, then by Lemma \ref{lem_low}, it is already $(q_0+1)$-inconsistent and in particular $q$-inconsistent. Then the $p$-point set $(b_i:i<p)$ contradicts the $(p,q)$-property. This shows that no formula $\phi(x;b)$, $b\in \psi(\monster)$ divides over $M$.
\end{proof}

As was already observed in \cite{ExtDef2}, the $(p,q)$-theorem has the following model-theoretic consequence.

\begin{Prop}\label{prop_corpq}
Fix a model $M$ of $T$ and let $\phi(x;y), \psi(y)$ be two formulas such that whenever $b\in \psi(\monster)$, $\phi(x;b)$ does not divide over $M$. Assume that $\phi(x;y)$ is NIP. Then there are finitely many global types $p_0,\ldots,p_{N-1}\in S_x(\monster)$ such that, for any $b\in \psi(\monster)$, $\phi(x;b)$ is in one of them.
\end{Prop}
\begin{proof}
Let $q> VC^*(\phi)$. By the previous lemma, there is $p$ such that the family $\{\phi(\monster;b):b\in \psi(\monster)\}$ has the $(p,q)$-property. Let $N$ be given by applying Fact \ref{fact_pq} to this pair $(p,q)$. Consider the partial type $$q(x_0,\ldots,x_{N-1})=\left\{\bigvee_{i<N} \phi(x_i;b) : b\in \psi(\monster)\right\}.$$

By construction of $N$, every finite subset of $q(x_0,\ldots,x_{N-1})$ is consistent. Hence the whole type is consistent and we obtain what we want.
\end{proof}

Our aim now is to give a model theoretic proof of this proposition.
Note that we do not say in the statement that the integer $N$ depends only on $p$, $q$ and $\phi(x;y)$, although it follows from the proof. The reason is that we will not manage to give a model-theoretic proof of that. The reader should convince herself that  the uniformity of $N$ does not follow by a simple compactness argument. The same problem appears in the proof of uniformity of honest definitions \cite{ExtDef2} and was in fact solved using the $(p,q)$-theorem.

\begin{Prop}\label{prop_pqequiv}
(Not using the $(p,q)$-theorem) Proposition \ref{prop_corpq} is equivalent to the following statement:

$(*)\quad $ Let $q\in S_y(M)$ and let $\tilde q$ be a global coheir of $q$. Assume that for $b\models q$, $\phi(x;b)$ does not divide over $M$. Then there is some $a\in \monster$ such that $\tilde q\vdash \phi(a;y)$.
\end{Prop}
\begin{proof}
Assume that $(*)$ holds and take $\phi(x;y)$, $\psi(y)$ over $M$ satisfying the hypothesis of Proposition \ref{prop_corpq}. Let $K\subseteq S_y(\monster)$ be the space of global types finitely satisfiable in $\psi(M)$. Then for every $\tilde q\in K$, $(*)$ gives us some $a\in \monster$ such that the clopen subset $\phi(a;y)$ of $K$ contains $\tilde q$. By compactness of $K$, we can find finitely many tuples $a'_1,\ldots,a'_n$ such that the sets $\phi(a'_i;y)$ cover $K$. In particular, they cover the set of types realised in $\psi(M)$. Hence we have found $a'_1,\ldots,a'_n$ such that for any $b\in \psi(M)$, one of $\phi(a'_i;b)$ holds. To get global types, simply let $(a_1,\ldots,a_n)$ realise an heir of $\tp(a'_1,\ldots,a'_n/M)$ over $\monster$. The heir property ensures that for any $b\in \psi(\monster)$, one of $\phi(a_i;b)$ holds. Now set $p_i=\tp(a_i/\monster)$.

Conversely, assume that Proposition \ref{prop_corpq} holds. Let $M$, $q$ and $\tilde q$ as in $(*)$. By Corollary \ref{cor_low}, there is some $\psi(y)\in \tp(b/M)$ such that $\phi(x;b')$ does not divide over $M$ whenever $b'\models \psi(y)$. Proposition \ref{prop_corpq} then gives types $p_0,\ldots,p_{n-1}\in S_x(\monster)$ such that for any $b'\in \psi(\monster)$, $\phi(x;b')$ is in one of them. For $i<n$, let $a_i\models p_i|_M$. Then any tuple $b'\in \psi(M)$ satisfies $\bigvee_{i<n} \phi(a_i;y)$. Hence $\tilde q$ also satisfies that formula which proves $(*)$.
\end{proof}

\subsection{Converging subsequences}

A sequence $(a_i:i<\omega)$ is called \emph{converging} if for any formula $\phi(x)\in L(\monster)$ the truth value of $\phi(a_i)$ is eventually constant.

A sequence is \emph{eventually indiscernible} if for any finite $\Delta$, some final segment of it is $\Delta$-indiscernible. By Ramsey and a simple diagonal argument, if the language is countable, then every sequence (indexed by $\omega$) has an eventually indiscernible subsequence. For our purposes, eventually indiscernible sequences are as good as truly indiscernible ones. Indeed if $\phi(x;y)$ is NIP and $(a_i:i<\omega)$ is eventually indiscernible, then for any $b$, the truth value of $\phi(a_i;b)$ alternates only finitely often (because by NIP and compactness, this is true for any $\Delta$-indiscernible sequence, for some finite $\Delta$ depending only on $\phi$).

The following is a very simple observation that will permit us to use sequences indexed by $\omega$, when in previous works we needed indiscernible sequences.

\begin{lemme}
A formula $\phi(x;y)$ is NIP if and only if for any model $M$, any sequence $(p_i:i<\omega)$ of $\phi$-types in $S_\phi(M)$ has a converging subsequence.
\end{lemme}
\begin{proof}
Assume that $\phi(x;y)$ is NIP. Let $L_0$ be a countable sublanguage containing $\phi$ and we work in $L_0$. Pick some $a_i\models p_i$. There is a subsequence $(a_{f(i)}:i<\omega)$ which is eventually indiscernible. Then as $\phi(x;y)$ is NIP the sequence $\tp_{\phi}(a_{f(i)}/\monster)$ is converging and so is the sequence $(p_{f(i)}:i<\omega)$.

Conversely, if $\phi(x;y)$ is not NIP, then there are an indiscernible sequence $(a_i:i<\omega)$ and $b$ is such that $\phi(a_i;b)$ holds if and only if $i$ is even. Then no subsequence of $(\tp_{\phi}(a_i/\monster):i<\omega)$ can be converging (since by indiscernibility, we can find a corresponding $b'$ such that $\phi(x;b')$ alternates on it).
\end{proof}

Note that this statement looks even more natural in continuous model theory, where it would probably make a convenient definition of NIP.

Actually, a little bit more is true when $M$ is countable: closures of sets are given by limits of sequences (we say that the space $S_\phi(M)$ is \emph{Fr\' echet}), as we will see now.

\begin{lemme}\label{lem_frechet}
Let $M$ be countable and $\Delta=\{\phi_i(x;y_i)\}$ a countable set of NIP formulas. Let $q$ be a global $\Delta$-type finitely satisfiable in $M$. Then there is a converging sequence $(b_i:i<\omega)$ of points in $M$ such that $\lim(\tp_{\Delta}(b_i/\monster))=q$.
\end{lemme}
\begin{proof}
Let $L_0$ be a countable language containing $\Delta$ and we work in $L_0$. Extend $q$ to some complete type $\tilde q$ finitely satisfiable in $M$. Let $I=(b'_i:i<\omega)$ be a Morley sequence of $\tilde q$ over $M$. By a diagonal construction, we can find a sequence $(b_i:i<\omega)$ in $M$ such that for any formula $\phi(x)\in \tilde q|_{MI}$, $\phi(b_i)$ is true for almost all $i$. We show that the sequence $(\tp_{\Delta}(b_i/\monster):i<\omega)$ converges to $q$. Let $q'$ be an accumulation point of $(\tp(b_i/\monster):i<\omega)$ in $S(\monster)$ (hence $q'$ is a global type finitely satisfiable in $M$). Then the restrictions of $q'$ and $\tilde q$ to $MI$ coincide. By an easy induction, this implies that the Morley sequence of $q'$ is the same as that of $\tilde q$. By Fact \ref{fact_morleyloc}, the restrictions of $\tilde q$ and $q'$ to $\Delta$-formulas agree.
\end{proof}

When $M$ is uncountable, one cannot work with sequences anymore. We have to replace them by more complicated directed families. Let $\kappa>\aleph_0$ be a cardinal. Let $S_{<\omega}(\kappa)$ be the set of finite subsets of $\kappa$ and let $\mathcal F$ be the filter on $S_{<\omega}(\kappa)$ generated by the sets $T_J=\{I\in S_{<\omega}(\kappa): I\supseteq J\}$ where $J$ ranges in $S_{<\omega}(\kappa)$.

\begin{lemme}\label{lem_frechetuncount}
Let $M$ have cardinality $\kappa$ and let $\Delta=\{\phi_i(x;y_i)\}$ be a set of NIP formulas of size $\leq \kappa$. Let $q$ be a global $\Delta$-type finitely satisfiable in $M$. Then there is a family $(b_l:l\in S_{<\omega}(\kappa))$ of points in $M$ such that $\lim_{\mathcal F}(\tp_{\Delta}(b_l/\monster))=q$.
\end{lemme}
\begin{proof}
The conclusion means that for any formula $\phi(x;c)\in q$, for $\mathcal F$-almost all $l\in S_{<\omega}(\kappa)$, $b_l \models \phi(x;c)$.

Extend $q$ to some complete type $\tilde q$ finitely satisfiable in $M$ and let $I=(b'_i:i<\omega)$ be a Morley sequence of $\tilde q$ over $M$. List the formulas in $\tilde q|_{MI}$ as $(\phi_i(x;c_i):i<\kappa)$. For $l\in S_{<\omega}(\kappa)$, take $b_l\in M$ realising $\bigwedge_{i\in l} \phi_i(x;c_i)$. Assume that the family $(\tp_{\Delta}(b_l/\monster):l\in S_{<\omega}(\kappa))$ does not converge to $q$ along $\mathcal F$ and let $\phi(x;c)\in q$ witness it. Let $\mathcal D$ be an ultrafilter on $S_{<\omega}(\kappa)$ extending $\mathcal F$ and containing $\{l\in S_{<\omega}(\kappa):\models \neg \phi(b_l;c)\}$. Let $q'=\lim_{\mathcal D}(\tp(b_l/\monster))$. Then as $\mathcal D$ contains $\mathcal F$, we have $q'|_{MI}=\tilde q|_{MI}$ and hence $q'^{(\omega)}|_M=\tilde q^{(\omega)}|_M$ as before. By Fact \ref{fact_morleyloc}, $\tilde q$ and $q'$ agree on $\Delta$-formulas, but this is a contradiction since $q'\vdash \neg \phi(x;c)$.
\end{proof}

Now, we are ready to prove $(*)$ of Proposition \ref{prop_pqequiv}.

\begin{thm}\label{th_coheirpq}
Assume that the formula $\phi(x;y)$ is NIP and $\phi(x;b)\in L(\monster)$ does not divide over $M$. Let $q$ be a coheir of $\tp(b/M)$. Then for some $a\in \monster$ $q\vdash \phi(a;y)$.
\end{thm}
\begin{proof}
To simplify the exposition, we first assume that $M$ is countable (which is all that is needed for Proposition \ref{prop_pqequiv}). We can then restrict to a countable sublanguage, and assume that $L$ is countable.

Let $q$ be a coheir of $\tp(b/M)$ and fix some $N$ containing $M$ and $|M|^+$-saturated. Let $I=(b'_i:i<\omega)$ in $N$ be a Morley sequence of $q$ over $M$ and let $b'_*\in N$ realise $q$ over $MI$. Let $\pi(x)=\bigwedge_{i<\VC^*(\phi)+1} \phi(x;b'_i)$. By the non-dividing assumption, $\pi(x)$ is consistent. Write $q|_{MI}(y)=\bigcup_{n<\omega} q_n(y)$ where each $q_n$ is finite. We now try to build a sequence $(b_i:i<\omega)$ of points in $M$ such that for each $n$:

$\bullet_1$ $\pi(x) \cup \{ \phi(x;b_{i-1})\leftrightarrow \neg \phi(x;b_{i}) : 0<i<n\}$ is consistent;

$\bullet_2$ $b_n\models q_n(y)$.

If we succeed, then by $\bullet_2$ and the argument in the proof of Lemma \ref{lem_frechet}, the sequence $(\tp_{\phi^{opp}}(b_i/\monster):i<\omega)$ is converging and this contradicts $\bullet_1$. Hence we must be stuck at some finite stage. Suppose we have built $(b_i:i<n)$ but cannot find $b_n$.

\vspace{4pt}
\underline{Claim}: There is $a_*$ realising $\pi(x) \cup \{ \phi(x;b_{i-1})\leftrightarrow \neg \phi(x;b_{i}) : 0<i<n\}$ and such that $\phi(a_*;b'_*)$ holds.

\emph{Proof}: Let $\psi(x)=\bigwedge_{0<i<n} \phi(x;b_{i-1})\leftrightarrow \neg \phi(x;b_{i})$ and $\phi'(x;y)=\phi(x;y)\wedge \psi(x)$. By construction, the partial type $\{\phi'(x;b'_i) : i<\VC^*(\phi)+1\}$ is consistent. We of course have $\VC^*(\phi')\leq \VC^*(\phi)$ since $y$ does not appear in $\psi(x)$. Furthermore, the sequence $(b'_i:i<\omega)+(b'_*)$ is indiscernible over the parameters of $\phi'(x;y)$. We then conclude from Lemma \ref{lem_low} that the set $\{\phi'(x;b'_i) : i<\omega\} \cup \{\phi'(x;b'_*)\}$ is consistent and this gives what we want.\hfill$\dashv$

\smallskip
Let $\pi_n(x)$ be the $\phi$-type of $a_*$ over $\{b_i:i<n\}$ (a finite partial type). Then for some $\epsilon\in \{0,1\}$, there is no $b\in M$ satisfying
$$q_n(y) \wedge (\exists x)(\pi(x)\wedge \pi_n(x) \wedge \phi(x;y)^{\epsilon}).$$

As $\tp(b'_*/MI)$ is finitely satisfiable in $q_n(M)$, \[\models \neg (\exists x )(\pi(x)\wedge \pi_n(x)\wedge \phi(x;b'_*)^{\epsilon}).\] As $a_*$ realises $\pi(x)\wedge \pi_n(x)$ and $\phi(x;b'_*)$, we must have $\epsilon=0$. We conclude that for any $b\in q_n(M)$ we have $\models \phi(a_*;b)$. As $q$ is finitely satisfiable in $q_n(M)$, we have $q\vdash \phi(a_*;y)$ which proves the proposition.

\smallskip
The proof for uncountable $M$ (and $L$) is similar, but using the idea of the proof of Lemma \ref{lem_frechetuncount} instead of Lemma \ref{lem_frechet}. We take $I$ and $b'_*$ exactly as above and define $\pi(x)$ in the same way. Enumerate the formulas in $q|_{MI}(y)$ as $\{\theta_i(y) : i<\kappa\}$. We try to build a family $(b_I:I\in S_{<\omega}(\kappa))$ of points in $M$ by induction on $|I|$ such that for each $I\in S_{<\omega}(\kappa)$, $I=\{k_0,\ldots,k_{n-1}\}$ listed in increasing order:

$\bullet'_1$ the type $\pi(x) \cup \{ \phi(x;b_{J_{i-1}})\leftrightarrow \neg \phi(x;b_{J_i}) : 0<i\leq n\}$ is consistent, where $J_i=\{k_0,\ldots,k_{i-1}\}$;

$\bullet'_2$ $b_I\models \bigwedge_{i\in I} \theta_i(y)$.

If we succeed, then by $\bullet'_2$ and the argument in the proof of Lemma \ref{lem_frechetuncount}, $\lim_{\mathcal F}(\tp_{\phi^{opp}}(b_I/\monster):I\in S_{<\omega}(\kappa))$ exists and this contradicts $\bullet'_1$. Let $I$ be such that $b_J$, for $J\subset I$ have been defined, but we cannot find $b_I$. Let $n=|I|$ and define the increasing sequence $\emptyset=J_0\subset J_1\subset \cdots \subset J_n=I$ of initial segments as in $\bullet'_1$. Then the type $\pi(x) \cup \{ \phi(x;b_{J_{i-1}})\leftrightarrow \neg \phi(x;b_{J_i}) : 0<i< n\}$ is consistent. For $i<n$, define $b_i=b_{J_i}$ and set $q_n(y)=\bigwedge_{i\in I}\theta_i(y)$. Now the end of the proof above goes through word for word.
\end{proof}

\subsection{Local description of invariant $\phi$-types}\label{sec_localdescription}

An interesting consequence of Theorem \ref{th_coheirpq} is that one can view an invariant $\phi$-type as a kind of type on finitely-satisfiable $\phi^{opp}$-types. We explain this now.

First, let us contemplate the usual description of invariant types. Assume that $p$ is an $M$-invariant $\phi(x;y)$-type. Then it is completely described by the function: $d_p \phi : S_y(M) \to \{0,1\}$ defined by $d_p \phi (\tp(b/M))=\epsilon$, $b\in \monster$, if $p\vdash \phi(x;b)^{\epsilon}$. If $p$ is finitely satisfiable, then the function $d_p\phi$ actually factors through the space $S_{\phi^{opp}}(M)$ of $\phi^{opp}$-types over $M$. To see this, assume that $b$ and $b'$ of size $|y|$ have the same $\phi^{opp}$-type over $M$, then for any $a\in M$, we have $\models \phi(a;b)\leftrightarrow \phi(a;b')$. Therefore also $p\vdash \phi(x;b)\leftrightarrow \phi(x;b')$ as $p$ is finitely satisfiable in $M$. This is no longer true if $p$ is only assumed to be $M$-invariant. In fact, the function $d_p \phi$ may not even factor through $S_{L_0}(M)$ where $L_0$ is a sublanguage containing $\phi(x;y)$. For a simple example of this, take $T$ to be \textsf{DLO} in the language $L_0=\{\leq \}$. Let $D(x)$ be a new unary predicate and have $D(x)$ name a non-definable initial segment of the universe. Take $\phi(x;y)=x\leq y$ and let $M$ be any model. Let $p$ be the $M$-invariant $\phi$-type of an element $a$ such that $D(\monster)<a<\monster \setminus D(\monster)$. Then $p$ is definable but it does not map to an invariant type in the reduct to $L_0$.

We now present a slightly different way to describe invariant types which only involves the formula $\phi$ and does not depend on the rest of the language.

\begin{Prop}
Let $\phi(x;y)$ be an NIP formula, $M\prec^+ N$ and $p(x)$ an $M$-invariant $\phi$-type. Let $b,b'\in \monster$ such that both $\tp(b/N)$ and $\tp(b'/N)$ are finitely satisfiable in $M$ and $\tp_{\phi^{opp}}(b/N)=\tp_{\phi^{opp}}(b'/N)$. Then we have $p\vdash \phi(x;b) \iff p\vdash \phi(x;b')$.
\end{Prop}
\begin{proof}
Assume that for example $p\vdash \phi(x;b)\wedge \neg \phi(x;b')$. Let $q_0(y)=\tp(b/N)$ and $q_1(y)=\tp(b'/N)$. Define also $\tilde q_i$ to be the unique $M$-invariant global extension of $q_i$, $i=0,1$. Take $(b_0,b_1)\models \tilde q_0\otimes \tilde q_1|_M$. Then by $M$-invariance of $p$, we have $p\vdash \phi(x;b_0)\wedge \neg \phi(x;b_1)$. Now $\tilde q_0\otimes \tilde q_1$ is a coheir of $\tp(b_0,b_1/M)$. Hence by Theorem \ref{th_coheirpq}, there is $a\in N$ such that $\tilde q_0(y_0)\otimes \tilde q_1(y_1) \vdash \phi(a;y_0)\wedge \neg \phi(a;y_1)$. But by hypothesis on $b,b'$, $\tilde q_0$ and $\tilde q_1$ must agree on formulas of the form $\phi(a;y)$, $a\in N$. Contradiction.
\end{proof}

Now to any global $\phi$-type $p(x)$ invariant over $M$, one can associate a function $f_p$ from the space $S^{fs}_{\phi^{opp}}(\monster,M)$ of global $M$-finitely satisfiable $\phi^{opp}$-types to $\{0,1\}$ defined by $f_p(q)=\epsilon$ if for some (any) $b\in \monster$ such that $\tp(b/N)$ is finitely satisfiable in $M$ and extends $q|_N$, we have $p\vdash \phi(x;b)^{\epsilon}$. Let $q_0(y)\in S(M)$ be a complete type and $b\models q_0$. How does one know from the function $f_p$ if $p\vdash \phi(x;b)$? We can take any global coheir $\tilde q_0$ of $q_0$ and let $q$ be its restriction to $\phi^{opp}$-formulas. Then $p\vdash \phi(x;b)$ if and only if $f_p(q)=1$. In particular, this shows that $f_p$ must take the same value on all types $q'$ such that $q'(y)\cup q_0(y)$ is finitely satisfiable in $M$. Also this shows that $p$ is entirely determined by the function $f_p$.

What is the range of the map $p\mapsto f_p$? This of course depends on the language, since increasing the language may add new invariant $\phi$-types. However it is possible to describe exactly which functions $f$ can arise as a function $f_p$ in some expansion of the structure. Let $\Omega$ be the set of all functions $S^{fs}_{\phi^{opp}}(\monster, M) \to \{0,1\}$ equipped with the product topology. Let also $Inv_\phi(M)\subseteq \Omega$ be the set of functions $f \in \Omega$ such that $f=f_p$ for some $M$-invariant $\phi$-type $p$.

\begin{lemme}
The set $Inv_\phi(M)$ is closed in $\Omega$.
\end{lemme}
\begin{proof}
Let $N\succ M$ be $|M|^+$-saturated and take $f\in \Omega \setminus Inv_\phi(M)$. Consider the partial type $\pi(x)=\{ \phi(x;b)^{\epsilon} : \tp(b/N)$ finitely satisfiable in $M$, $\epsilon=f(\tp_{\phi^{opp}}(b/N))\}$. If $f$ is not in $Inv_\phi(M)$, then $\pi(x)$ does not extend to a global $M$-invariant type. By compactness, some finite part of it does not extend to an $M$-invariant type. This shows that some neighborhood of $f$ in $\Omega$ is disjoint from $Inv_\phi(M)$ and hence $Inv_\phi(M)$ is closed.
\end{proof}

Let $s\in S_{\phi}(M)$ be a $\phi$-type over $M$. Define the function $\hat f_s : S^{fs}_{\phi^{opp}}(\monster, M) \to \{0,1\}$ by $\hat f_s(q) = \epsilon$ if $q\vdash \phi(a;y)^\epsilon$ when $a\models s$. Using $s\to \hat f_s$, we identify $S_\phi(M)$ to a subset of $\Omega$ (the induced topology on $S_\phi(M)$ is in general strictly finer than its natural one).

\begin{Prop}
With notations as above, any function $f_p$ lies in the closure of $S_\phi(M)\subseteq \Omega$.
\end{Prop}
\begin{proof}
Let $q_0,\ldots,q_{n-1}$ be $M$-finitely satisfiable $\phi^{opp}$-types. Take $f_p \in \Omega$. We have to find a type $s\in S_\phi(M)$ such that $\hat f_s$ agrees with $f_p$ on $\{q_0,\ldots,q_{n-1}\}$. Pick complete expansions $\tilde q_0,\ldots,\tilde q_{n-1}$ of $q_0,\ldots,q_{n-1}$ respectively, finitely satisfiable in $M$. Let $\tilde q(y_0,\ldots,y_{n-1})=\tilde q_0(x_0)\otimes \cdots \otimes \tilde q_{n-1}(x_{n-1})$. Then $\tilde q$ is finitely satisfiable in $M$. Let $(b_0,\ldots,b_{n-1})\models \tilde q|_M$. Then the formula $$\psi(x;b_0,\ldots,b_{n-1})=\bigwedge_{i<n} \phi(x;b_i)^{f_p(q_i)}$$ is in $p$ and in particular does not divide over $M$. By Theorem \ref{th_coheirpq}, there is $a\in \monster$ such that $\tilde q\vdash \psi(a;\bar y)$. Set $s=\tp(a/M)$.
\end{proof}

We conclude by showing that in some expansion $M'$ of $M$, $Inv_\phi(M')$ is exactly the closure of $S_\phi(M)$ inside $\Omega$.

\begin{Prop}
Let $M'$ be the expansion of $M$ where we have added a predicate to name any externally definable set of the form $\phi(M;b)$, $b\in \monster$. Then $Inv_\phi(M')$ is equal to the closure of $S_\phi(M')=S_\phi(M)$ in $\Omega$.
\end{Prop}
\begin{proof}
Using the two previous lemmas, it suffices to show that any function of the form $\hat f_s$, $s\in S_\phi(M)$ corresponds to an invariant $\phi$-type in the sense of $M'$. But we have done exactly what is needed to ensure that all $\phi$-types over $M$ are definable. Hence $\hat f_s=f_p$ where $p$ is the global definable extension of $s$.
\end{proof}

\subsection{Definable $(p,q)$}

The following is a definable version of Proposition \ref{prop_corpq}, which was conjectured in \cite{ExtDef2} (assuming the full theory is NIP), in fact before the link with the $(p,q)$-theorem was noticed.

\begin{conj}\label{conj_pq}
Assume that the formula $\phi(x;y)$ is NIP. Let $M$ be a model and $\phi(x;b)\in L(\monster)$ non-dividing over $M$. Then there is a formula $\psi(y)\in \tp(b/M)$ such that the partial type $\{\phi(x;b):b\in \psi(M)\}$ is consistent.
\end{conj}

Conjecture \ref{conj_pq} reduces to the case where $L$ is countable because we can take a countable sublanguage containing $\phi(x;y)$. Then we may assume that $M$ is countable, because if $\phi(x;b)$ does not divide over $M$, then there is a countable $M'\prec M$ over which $\phi(x;b)$ does not divide.
We will now prove this conjecture assuming that $L$ and $M$ are countable and that $\tp(b/M)$ has only countably many coheirs.

\begin{lemme}
($L$ is countable) Conjecture \ref{conj_pq} is equivalent to the following statement:

$(**)\quad $ Let $M$ be a countable model, $q\in S_y(M)$. Assume that for $b\models q$, $\phi(x;b)$ does not divide over $M$. Then there is some $a\in \monster$ such that for any global coheir $\tilde q$ of $q$, $\tilde q\vdash \phi(a;y)$.
\end{lemme}
\begin{proof}
The reduction to a countable $M$ is explained above.

Assume that $(**)$ holds and let $q$ be such that $\phi(x;b)$ does not divide over $M$ whenever $b\models q$. Let $a$ be given by $(**)$. Then the set $\phi(a;M)$ must contain a subset $\psi(M)$, with $\psi(y)\in q$: Assume not, then we can easily build a sequence $(b_i:i<\omega)$ of points of $M\setminus \phi(a;M)$ such that $\tp(b_i/M)$ converges to $q$. For any non-principal ultrafilter $\mathcal D$ on $\omega$, the limit type $\lim_{\mathcal D} \tp(b_i/\monster)$ is a global coheir of $q$ which does not satisfy $\phi(a;y)$.

Conversely, if the conjecture holds, let $\psi(y)$ be given by it. Then pick $a\in \monster$ such that $\phi(a;b)$ holds for all $b\in \psi(M)$. The formula $\phi(a;y)$ is in every coheir of $q$.
\end{proof}

In the following theorem, we assume NIP of the whole theory and not only of one formula $\phi(x;y)$.
\begin{thm}\label{th_pqdef}
Assume that $L$ is countable and $T$ is NIP. Let $M$ be a countable model and $\phi(x;b)\in L(\monster)$ non-dividing over $M$. Assume that $q=\tp(b/M)$ has only countably many global coheirs. Then there is $a\in \monster$ such that $\phi(a;y)$ is in every global coheir of $q$.
\end{thm}
\begin{proof}
Fix some $N$ containing $M$ and $|M|^+$-saturated. Let $(b'_i:i<\omega)$ in $N$ be a strict Morley sequence in $\tp(b/M)$ over $M$ and $\pi_0(x)=\{\phi(x;b'_i):i<\VC^*(\phi)+1\}$. Then, by the facts recalled in Section \ref{sec_fork}, for any $\psi(x)\in L(M)$, if $\pi_0(x)\wedge \psi(x)$ is consistent, then it is non-forking over $M$.

Let $q^0,q^1,\ldots$ list the coheirs of $q=\tp(b/M)$. Write $q^0|_{MI}(y)=\bigcup_{n<\omega} q^0_n(y)$ where each $q^0_n$ is finite, and build as in Theorem \ref{th_coheirpq} a maximal sequence $(b_i:i<n)$ of points in $M$ such that:

$\bullet_1$ $\pi_0(x) \cup \{ \phi(x;b_{i-1})\leftrightarrow \neg \phi(x;b_{i}) : 0<i<n\}$ is consistent.

$\bullet_2$ $b_i\models q_i^{0}(y)$.

Let $\pi'_1(x)=\pi_0(x) \cup \{ \phi(x;b_{i-1})\leftrightarrow \neg \phi(x;b_{i}) : 0<i<n\}$. Then $\pi'_1(x)$ is consistent and thus does not fork over $M$. Let $b'_*$ realise $q^0_n|_{MI}$ and $a_*\models \pi'_1(x)$ such that $\tp(a_*/Nb'_*)$ is $M$-invariant. Then as $\phi(a_*;b'_1)$ holds, also $\phi(a_*;b'_*)$ holds. We then conclude as in Theorem \ref{th_coheirpq} that for any $b\in q_n^0(M)$, we have $\models \phi(a_*;b)$.

Set $\pi_1(x)$ to be the union of $\pi_0(x)$ and the $\phi$-type of $a_*$ over $\{b_i:i<n\}$ and iterate the construction with $q^1$ instead of $q^0$ and $\pi_1(x)$ instead of $\pi_0(x)$. After $\omega$ steps, we obtain a consistent partial type $\pi_\omega(x)=\bigcup \pi_n(x)$ over $N$ with the property that if $a_*\models \pi_\omega(x)$ and $b_*\models q^m|_{Ma_*}$ for some $m$, then $\phi(a_*;b_*)$ must hold. Hence $\phi(a_*;y)$ is in $q^m$ and we have what we were looking for.
\end{proof}

Actually, this proof also works if the space of global coheirs of $q$ is only assumed to be separable. Take the $q^n$'s to form a dense set and build $\pi_\omega$ as above. Let $a\models \pi_\omega$. Then the formula $\neg \phi(a;y)$ defines an open set in the space of coheirs of $q$. We know that it cannot contain any of the $q^n$'s. Hence it is empty.

\begin{qu}
Let $M$ be a countable pseudofinite NIP structure on a countable language, then does every type over $M$ have countably many coheirs?
\end{qu}

\section{The $F_M$ retraction}\label{sec_fm}

We now move to another topic, which is only loosely connected to the previous section. Our aim is to construct and study a canonical retraction from the space of $M$-invariant types to that of $M$-finitely satisfiable types. It is still slightly mysterious why such a retraction exists, but it turns out to be rather useful. It will permit us to give a more conceptual proof of the dichotomy proved in \cite{invdp}: an $M$-invariant dp-minimal type is either definable or finitely satisfiable in $M$. Also, it is used in \cite{ChePilSim} with A. Chernikov and A. Pillay to study how some notions concerning groups are invariant under naming externally definable sets.

\smallskip\noindent
\textbf{Assumption}: Throughout this section, we assume that $T$ is NIP.

\subsection{The construction}

Let $L_{\pred}=L\cup \{\pred(x)\}$ where $\pred$ is a new unary predicate. An expansion of a model of $T$ to $L_{\pred}$ is called a pair. We will always write pairs as $(M,A)$, where $M$ is the universe of the structure and $A=\pred(M)$ is the interpretation of $\pred$ in $M$. Let $M\models T$ and let $M\prec^+ M'$ and consider the pair $(M',M)$. Let also $(N',N)$ be an $|M'|^+$-saturated elementary extension of $(M',M)$.

Let $p$ be a global $M$-invariant type (in the language $L$).

\smallskip \noindent
\underline{Claim}: The partial $L_{\pred}$-type $p(x)|_N \cup \pred(x)$ implies a complete $L$-type over $M'$.

Proof: Assume not. Then for some $L$-formula $\phi(x;b)\in L(M')$, $p(x)|_N\cup \pred(x)$ is consistent both with $\phi(x;b)$ and with $\neg \phi(x;b)$. We can then construct inductively a sequence $(a_i:i<\omega)$ of points of $N$ such that:

$\cdot$ $a_{2i}\models p(x)|_{Ma_{<2i}} \cup\pred(x) \cup \phi(x;b)$;

$\cdot$ $a_{2i+1} \models p(x)|_{Ma_{<2i+1}} \cup \pred(x) \cup \neg\phi(x;b)$.

In the reduct to $L$, the sequence $(a_i:i<\omega)$ is a Morley sequence of $p$ and as such is $L$-indiscernible. The formula $\phi(x;b)$ alternates infinitely often on it and this contradicts NIP, proving the claim.\hfill$\dashv$

\smallskip
Therefore there is a unique $L$-type $q(x) \in S(M')$ such that $p(x)|_N \cup \pred(x) \cup q(x)$ is consistent. As $q(x)$ is consistent with $\pred(x)$ it must be finitely satisfiable in $M$. As $M'$ is $|M|^+$-saturated, $q$ extends uniquely to a global $M$-finitely satisfiable $L$-type $\tilde q$. We now define $F_M(p)$ to be equal to $\tilde q$. It is not hard to check that this is well defined, {\it i.e.}, depend on neither $M'$ nor $(N',N)$.

\begin{lemme}
Let $M\models T$, then the map $F_M$ from $M$-invariant types to $M$-finitely satisfiable types has the following properties:

(i) $F_M(p)|_M = p|_M$;

(ii) $F_M$ is continuous;

(iii) if $p$ is finitely satisfiable in $M$, then $F_M(p)=p$;

(iv) for any $M$-definable function $f$, $f_*(F_M(p))=F_M(f_*(p))$.
\end{lemme}
\begin{proof}
Point (i) is clear. Point (ii) is by compactness: given a formula $\phi(x;b)$ over $M'$, if $F_M(p)\models \phi(x;b)$, then there is some finite part $p_0$ of $p|_N$ such that $p_0(x)\cup \pred(x)$ implies $\phi(x;b)$.

If $p$ is finitely satisfiable in $M$, then $p(x)|_{N'}\cup \pred(x)$ is a consistent $L_\pred$-type, hence $F_M(p)=p$.

Point (iv) is easy to check from the definition.
\end{proof}

\begin{rem}
Beware that in general we do not have $F_M(p\otimes q)=F_M(p) \otimes F_M(q)$.
\end{rem}

\begin{ex}
Take $T$ to be DLO and $M\models T$. Let $a\in M$ and consider the type $p\in S(M)$ such that $p\vdash x>c \iff c\leq a$. Then $p$ has two global $M$-invariant extensions: a definable one $p_{def}$ and an $M$-finitely satisfiable one $p_{fs}$. Then $F_M(p_{def})=F_M(p_{fs})=p_{fs}$.
\end{ex}

\subsection{The reverse type}

We present another construction which associates a finitely satisfiable type to an invariant type, but this time over a larger base. It already appeared in \cite{distal}.

Let $p$ be an $M$-invariant type and let $M\prec^+ N$.

\smallskip
\underline{Claim}: Given any $b\in \monster$, there is some $B\subset N$ of size $|M|$ such that any two realisations of $p|_B$ in $N$ have the same type over $Mb$.

Proof: Assume not. Then we can build inductively a sequence $(a_{i,0}a_{i,1}:i<|M|^+)$ in $N$ such that $a_{i,0},a_{i,1}$ both realize $p$ over $Ma_{<i,0}a_{<i,1}$, but $a_{i,0}, a_{i,1}$ do not have the same type over $Mb$. Then for any $\eta: |M|^+ \to \{0,1\}$, the sequence $(a_{i,\eta(i)} : i<|M|^+)$ is indiscernible (it is a Morley sequence of $p$). Pruning the sequence, we may assume that for some formula $\phi(x)\in L(Mb)$, we have $\models \phi(a_{i,0})\wedge \neg \phi(a_{i,1})$ for all $i$. Taking $\eta$ to alternate infinitely often between 0 and 1, we contradict NIP.\hfill$\dashv$

\smallskip
We can now define a global type $R_N(p)$ as follows: Let $b\in \monster$ and take $B$ as given by the claim. Let $a$ be a realisation of $p|_B$ in $N$ and set $R_N(p)|_{Mb}=tp(a/Mb)$. It is easy to see that the different restrictions of $R_N(p)$ defined are compatible and thus we obtain a global type $R_N(p)$. By construction, this type is  $N$-finitely satisfiable and its restriction to $N$ coincides with $p|_N$.

We call $R_N(p)$ the \emph{reverse type} of $p$ over $N$. The reason for this terminology comes from point (i) below.

\begin{lemme}\label{lem_reverse}
The following facts hold:

(i) For any $M$-invariant type $q$, we have $$(R_N(p)_x\otimes q_y)|_N= (q_y\otimes p_x)|_N.$$

(ii) The type $R_N(p)$ commutes with any $M$-invariant type (in particular with $p$).
\end{lemme}
\begin{proof}
(i) Let $q$ be $M$-invariant and $b\models q|_N$. For any $B\subset N$ containing $M$ and $a\models p|_B$, we have $\tp(b,a/B)=q\otimes p|_B$. Hence, by construction of $R_N$, the stated equality holds.

(ii) Let $q$ be $M$-invariant; take some $d\in \monster$, $a_*\models R_N(p)|_{Nd}$ and $b\models q|_{Na_*d}$. Let $a\in N$ such that $\tp(a/Mdb)=R_N(p)|_{Mdb}$. Then $(a,b)\models (R_N(p)_x\otimes q_y)|_{Md}$. Also $a \equiv_{Md} a_*$, hence by $M$-invariance of $q$, $a \equiv_{Mdb} a_*$ and $(a,b)\models (q_y \otimes R_N(p)_x)|_{Md}$.
\end{proof}

We deduce from this that the Morley sequence of $R_N(p)$ has the same type over $N$ as the Morley sequence of $p$ read backwards:

\begin{lemme}\label{lem_revmorley}
Notations being as above, let $(a_0,\ldots,a_{n-1})$ be an initial segment of a Morley sequence of $p$ over $N$. Then $(a_{n-1},\ldots,a_0)$ is the beginning of a Morley sequence of $R_N(p)$ over $N$.
\end{lemme}
\begin{proof}
In what follows, $p^{(n)}(x_{n-1},\ldots,x_0)$ denotes the type of the first $n$ elements in a Morley sequence of $p$, but with decreasing indices, hence $x_{n-2}\models p|_{x_{n-1}}$, $x_{n-3}\models p|_{x_{n-1}x_{n-2}}$ etc.

We show the result by induction on $n$. We already know that $p|_N=R_N(p)|_N$ which gives the case $n=1$.

Assume we know the result for $n$, then \[\tp(a_n/Na_0\ldots a_{n-1})=p\upharpoonright Na_0\ldots a_{n-1}.\] Hence by induction, $\tp(a_n,\ldots, a_0/N)=p(x_n)\otimes R_N(p)^{(n)}(x_{n-1}\ldots x_0)|_N$ which is equal to $R_N(p)^{(n)}(x_{n-1}\ldots x_0) \otimes p(x_n)|_N$ by Lemma \ref{lem_reverse} (ii). As the restrictions of $p$ and $R_N(p)$ to $N$ agree, this last expression is equal to \[R_N(p)^{(n)}(x_{n-1}\ldots x_0) \otimes R_N(p)(x_n)|_N=R_N(p)^{(n+1)}(x_n,\ldots,x_0)|_N.\]
\end{proof}

Note that we have only used in the proof the fact that $R_N(p)$ and $p$ have the same restriction to $N$ and the fact that those two types commute. As an $N$-invariant type is determined by the type of its Morley sequence over $N$, we deduce the following lemma.

\begin{lemme}\label{lem_uniqueRn}
Let $M\pprec N$ and $p$ be $M$-invariant. Then $R_N(p)$ is the only $N$-invariant extension of $p|_N$ which commutes with $p$.
\end{lemme}

\subsection{Commutativity properties}

We want to argue that $F_M(p)$ somehow captures the finitely satisfiable part of $p$. We do not define the ``finitely satisfiable part of $p$", but we have in mind something like $\{ q\in S(\monster): q$ is $M$-finitely satisfiable and does not commute with $p\}$.

\begin{lemme}\label{lem_fsfm}
Let $p, q$ be $M$-invariant types, $q$ being finitely satisfiable in $M$. Then we have $F_M(q\otimes p)=q\otimes F_M(p)$.
\end{lemme}
\begin{proof}
Let $(M',M)\prec^{+} (N',N)$ be as in the definition of $F_M$. Take yet another extension $(N',N)\prec^{+} (N_1',N_1)$. Let $a\in N_1$ realise $p|_N$. As $q$ is finitely satisfiable in $M$, $q|_{N'a}\cup \pred(x)$ is consistent and is realised by some $b\in N_1$. Then by definition $F_M(q\otimes p)$ is given by $\tp(b,a/M')$. But $\tp(b/M'a)=q|_{M'a}$, hence the result.

\end{proof}

We know that a type is definable if and only if it commutes with all types finitely satisfiable in a small model. We show now that it is enough to check commutativity with one specific finitely satisfiable type.
\begin{Prop}\label{prop_deffm}
Let $p$ be an $M$-invariant type. If $p$ commutes with $F_M(p)$, then $p$ is definable.
\end{Prop}
\begin{proof}
Let $M\prec^+ N$ and consider $s_1$ and $s_2$ two $M$-invariant types. Take $a_1,a_2\in \monster$ then by Lemma \ref{lem_reverse} (i),

\smallskip
$\oplus_0$ ($a_1\models s_1|_N$ and $a_2\models s_2|_{Na_1}$) if and only if ($a_2\models s_2|_N$ and $a_1\models R_N(s_1)|_{Na_2}$).

\smallskip

In particular:

\smallskip
$\oplus_1$ If $a_2\models s_2|_N$ and $a_1\models R_N(s_1)|_{Na_2}$, then $\tp(a_1,a_2/N)$ is $M$-invariant.

\smallskip
Assume that $p$ commutes with $F_M(p)$ but is not definable. Then by Lemma \ref{lem_def} there is some type $q$ finitely satisfiable in $M$ such that $p_x\otimes q_y|_M \neq q_y\otimes p_x|_M$. Hence by $\oplus_0$, $p_x\otimes q_y|_M \neq R_N(p)_x\otimes q_y|_M$.

Let $b\models q|_N$ and build a maximal sequence $(c^0_i,c^1_i : i<\alpha)$ such that:

$\cdot$ $\tp(b+(c^1_i,c^0_i:i<\alpha) /N)$ is finitely satisfiable in $M$;

$\cdot$ $(c^0_i,c^1_i)\models R_N(F_M(p))^{(2)}\upharpoonright c^0_{<i}c^1_{<i}N$;

$\cdot$ $\tp(bc^0_i/M)\neq \tp(bc^1_i/M)$.

By NIP such a maximal sequence exists (for any $\eta: \alpha\to \{0,1\}$, the sequence $(c^{\eta(i)}_i:i<\alpha)$ is a Morley sequence of $R_N(F_M(p))$ and as such is indiscernible. If $\alpha\geq |M|^+$, then the last bullet contradicts NIP).

\smallskip
Let now $a^0\models R_N(p)$ over everything and $a^1\models p$ over everything. Then $\tp(ba^0/M)\neq \tp(ba^1/M)$.

Let $s$ be the type of $b+(c^1_i,c^0_i)_{i<\alpha}+a^0a^1$ over $N$. This type is $M$-invariant by $\oplus_1$. If we apply $F_M$ to it, then the restriction to the variables corresponding to $b+(c^0_i,c^1_i)_{i<\alpha}$ does not change since this type is $M$-finitely satisfiable. Thus we can find $(g^0,g^1)$ such that $b+(c^0_i,c^1_i)_{i<\alpha}+g^0g^1$ realises $F_M(s)$ over $N$.

By $\oplus_0$, $b\models q|_{Na^0}$. Then by Lemma \ref{lem_fsfm} and another application of $\oplus_0$, we have

\smallskip
$\oplus_2$ $g^0 \models R_N(F_M(p))\upharpoonright Mb$.

\smallskip
Next, notice that the tuple $c^0_{<\alpha}c^1_{<\alpha}$ realises a Morley sequence of $F_M(p)$ over $N$ (when ordered as $\ldots c^1_{i+1}c^0_{i+1}c^1_i c^0_i   \ldots$). As $p$ and $F_M(p)$ commute, we deduce that $p$ and $R_N(p)$ have the same restriction to $Nc^0_{<\alpha}c^1_{<\alpha}$. Therefore $a^1\models R_N(p)\upharpoonright  Nc^0_{<\alpha}c^1_{<\alpha}$.

Hence by the same argument as above, we conclude

\smallskip
$\oplus_3$ $g^1\models R_N(F_M(p))\upharpoonright  Nc^0_{<\alpha}c^1_{<\alpha}$.

\smallskip
Also, we still have $\tp(bg^0/M)\neq \tp(bg^1/M)$ since $F_M$ preserves types over $M$. So by $\oplus_2$, $g^1$ does not realise $R_N(F_M(p))$ over $Mb$. Now take $g$ satisfying $R_N(F_M(p))$ over everything. Set $c^0_\alpha=g^1$ and $c^1_\alpha=g$. This contradicts maximality of the sequence $(c^0_ic^1_i : i<\alpha)$.
\end{proof}

\begin{lemme}
Let $p$ be $M$-invariant, not finitely satisfiable in $M$. Let $I$ be a Morley sequence of $p$ over $M$, then $p|_{MI}$ is not finitely satisfiable in $M$.
\end{lemme}
\begin{proof}
Let $\mathcal D$ be an ultrafilter on $M$ whose limit type $q$ has the same restriction as $p$ over $MI$. One then shows inductively that $p$ and $q$ have the same Morley sequence, hence $p=q$.
\end{proof}

\begin{lemme}
Let $M\pprec N$. If $p$ is $M$-invariant, but not finitely satisfiable in $M$, then $F_N(p)$ does not commute with $R_N(p)$.
\end{lemme}
\begin{proof}
Let $I=(a_i:i<\omega)$ be a Morley sequence of $p$ over $N$. By the previous lemma, $p|_{NI}$ is not finitely satisfiable in $N$. Let $\phi(x;a_1,\ldots,a_n)$ be a formula witnessing it. By Lemma \ref{lem_revmorley}, we have $(a_n,\ldots,a_1)\models R_N(p)^{(n)}|_N$. As $\phi(N;a_1,\ldots,a_n)=\emptyset$, necessarily $F_N(p)\otimes R_N(p)^{(n)} \models \neg \phi(x;y_n,\ldots,y_1)$. On the other hand, $R_N(p)^{(n)}\otimes F_N(p)$ and $R_N(p)^{(n)} \otimes p$ have the same restriction to $N$, hence $R_N(p)^{(n)}\otimes F_N(p) \models \phi(x;y_n,\ldots,y_1)$. So $F_N(p)$ and $R_N(p)^{(n)}$ do not commute, hence already $F_N(p)$ and $R_N(p)$ do not commute.
\end{proof}

\begin{lemme}
The type $p$ is definable if and only if $F_{N}(p)=R_N(p)$.
\end{lemme}
\begin{proof}
Assume that $p$ is definable. Then we know that $p$ commutes with every $N$-finitely satisfiable type. Thus $F_N(p)$ and $R_N(p)$ are two $N$-invariant types extending $p|_N$ and commuting with $p$. By Lemma \ref{lem_uniqueRn}, there can be only one such type. Hence $F_N(p)=R_N(p)$.

Conversely, if $p$ is not definable, then $p$ does not commute with $F_{N}(p)$ by Proposition \ref{prop_deffm}. But $R_N(q)$ does commute with $p$, hence $F_N(p)\neq R_N(p)$.
\end{proof}

\begin{Prop}
Let $p$ be $M$-invariant and let $q$ be a type finitely satisfiable in $M$. If $q$ commutes with $F_M(p)$, then it commutes with $p$.
\end{Prop}
\begin{proof}
Let $M \prec^+ N$. Assume that $q$ does not commute with $p$. Replacing $q$ by $q^{(\omega)}$ and using Lemma \ref{lem_commseq}, we may assume that $q\otimes p|_M\neq p\otimes q|_M$. Let $r= p\otimes q^{(\omega)} \otimes p$ and let $(a_1,\bar b,a_0)\models F_M(r)|_N$. Then for every $k<\omega$, $\tp(a_1,b_k/M)\neq \tp(a_0, b_k/M)$. On the other hand by Lemma \ref{lem_fsfm}, $\tp(\bar b,a_0/N)=q^{(\omega)}\otimes F_M(p)|_N$. Thus for every $k$, $(b_k,a_1)$ does not realise $q\otimes F_M(p)$ over $M$, and $\tp(b_k,a_1/M)$ is constant as $k$ varies. If $q$ commutes with $F_M(p)$, this contradicts Proposition \ref{prop_finitecofinite}.
\end{proof}

Note that you cannot expect the other implication (if $q$ commutes with $p$, then it commutes with $F_M(p)$). For example if $p$ is definable not finitely satisfiable in $M$, then $F_M(p)$ commutes with $p$, but does not commute with itself.

\subsection{Application to dp-minimal types}

Let $p$ be $M$-invariant and $M\prec^+ N$. Summarising some of the results above, the situation is as follows:

$\cdot$ $p$ and $R_N(p)$ commute;

$\cdot$ $F_N(p)$ commutes with $R_N(p)$ if and only if $p$ is finitely satisfiable;

$\cdot$ $F_N(p)$ commutes with $p$ if and only if $p$ is definable.

\smallskip
We have now all we need to give another, more conceptual, proof of the dichotomy for dp-minimal types proved in \cite{invdp}.

Recall that a type $p$ is \emph{dp-minimal} if for any $A$ and any two sequences $I$ and $J$  mutually indiscernible over $A$, for any $a\models p$, either $I$ or $J$ is indiscernible over $Aa$. In particular if $p$ is $M$-invariant and dp-minimal, $q$ and $r$ are two $M$-invariant types which commute, then $p$ commutes with either $q$ or $r$. To see this let $M\prec^+N$ and build $(\bar b,\bar c,a,\bar b',\bar c')\models q^{(\omega)}\otimes r^{(\omega)} \otimes p \otimes q^{(\omega)} \otimes r^{(\omega)}$ over $N$. The sequences $\bar b+\bar b'$ and $\bar c+\bar c'$ are mutually indiscernible but none is indiscernible over $Na$.

\begin{thm}[\cite{invdp}, Theorem 2.6]\label{th_dpmindich}
Let $p$ be $M$-invariant and dp-minimal, then $p$ is either finitely satisfiable in $M$ or definable.
\end{thm}
\begin{proof}
If $p$ is dp-minimal, then already $p|_N$ is dp-minimal. Hence $F_N(p)$ is dp-minimal. Since $R_N(p)$ and $p$ commute, by dp-minimality, $F_N(p)$ must commute with one or the other. From the observations above, we deduce that $p$ is either finitely satisfiable or definable.
\end{proof}

\section{Amalgamation of invariant types}\label{sec_amalg}

The results presented in this section are independent of the rest of the paper. They deal with amalgamating invariant types in NIP theories. They were used in previous, more complicated, proofs of some of the results here and we hope that they might turn out to be useful elsewhere.

\smallskip \noindent
\textbf{Assumption}:  Throughout this section, we assume that $T$ is NIP.

\begin{lemme}\label{lem_amalg}
Let $M\pprec N$ and let $p(x),q(y)$ be $M$-invariant types. Let $a\models p|_N$ and $b\models q|_N$, then there is some $N$-invariant type $r(x,y)$ extending $p(x)\cup q(y)\cup \tp(a,b/N)$.
\end{lemme}
\begin{proof}
We know by Fact \ref{fact_fork} (or see Corollary 3.34 of \cite{CherKapl}) that any $N$-invariant consistent partial type extends to a global $N$-invariant type. Thus it is enough to show that $p(x)\cup q(y)\cup \tp(a,b/N)$ is consistent. This is easy: any inconsistency can be dragged down in $N$ by $M$-invariance of $p$ and $q$.
\end{proof}

We recall honest definitions. If $A\subseteq M$, then the pair $(M,A)$ is the expansion of $M$ obtained by adding a unary predicate $\mathbf P(x)$ naming the subset $A$.

\begin{fait}[\cite{NIPbook}, Theorem 3.13]\label{fact_honest}
Let $M\models T$, $A\subseteq M$, $\phi(x;y) \in L$ and $c\in M$ a $|y|$-tuple. Assume that $\phi(x;y)$ is NIP. Then there is an elementary extension $(M,A)\prec (M',A')$, a formula $\phi'(x;z)\in L$ and a tuple $c'$ of elements of $A'$ such that \[\phi(A;c)= \phi'(A;c') \text{  and  } \phi'(A';c') \subseteq \phi(A';c).\]
\end{fait}

\begin{Prop}
Let $M\prec^+ N$ and let $p,q$ be two global types finitely satisfiable in $M$. Let $a\models p|_N$ and $b\models q|_N$, then $\tp_{x,y}(a,b/N)\cup p(x)\cup q(y)$ is finitely satisfiable in $N$.
\end{Prop}
\begin{proof}
Otherwise, there is some $\theta(x;y)\in L(N)$, $\phi(x;c)\in L(\monster)$ and $\psi(y;c)\in L(\monster)$ such that $\phi(x;c)\wedge \psi(y;c)\wedge \theta(x;y)$ has no solution in $N$. Let $(\monster,N)$ denote the expansion of $\monster$ obtained by adding a unary predicate $\mathbf P(x)$ to name the submodel $N$. Let $(\monster,N)\prec^+ (\monster',N')$. By Fact \ref{fact_honest}, we can find two formulas $\phi'(x;c')\in L(N')$ and $\psi'(y;c')\in L(N')$ such that $\phi(N;c)= \phi'(N;c')$, $\phi'(N';c')\subseteq \phi(N';c)$ and same for $\psi'$. It follows that $\phi'(x;c')\wedge \psi'(y;c')\wedge \theta(x;y)$ has no solution in $N'$, hence is inconsistent. As $p$ is finitely satisfiable in $N$ (indeed in $M$) and $p\vdash \phi(x;c)$, necessarily $p\vdash \phi'(x;c')$. Similarly, $q\vdash \psi'(x;c')$. Hence we conclude that $\theta(x;y)\wedge p(x) \wedge q(y)$ is inconsistent. But this contradicts Lemma \ref{lem_amalg}.
\end{proof}

The assumption that $N$ is saturated over $M$ is necessary to avoid situations such as the following: let $M=(\mathbb Q;<)$. Let $p(x)$ be the $M$-invariant global type defined by $p\vdash x<a \iff 0<a$ and $p\vdash 0<x$ and let $q(y)$ be defined by $q\vdash y<a \iff 0<r<a$ for some $r\in \mathbb Q$. Then $p(x)|_M\cup q(y)|_M \cup \{x>y\}$ is consistent, but we cannot amalgamate $p(x)\cup q(y) \cup \{x>y\}$.

\smallskip
Another situation where NIP allows us to amalgamate is when the types considered fit inside indiscernible sequences. This was observed in \cite{distal}. We recall (a special case of) Lemma 2.9 from that paper which will allow us to work over slightly more complicated bases at the expense of commutativity assumptions.

Recall the following notation: if $I=(a_t : t\in \mathcal I)$ is an indiscernible sequence with no last element, then $\lim(I/A)$ denotes the type $\lim(\tp(a_t/A))$, which exists by NIP.

\begin{lemme}\label{lem_strongbasechange}
Let $I_1,I_2,I_3$ be sequences of tuples, without endpoints. Assume that the concatenation $I_1+I_2+I_3$ is indiscernible over $A$. Let $a,b\in \monster$ such that $I_1+a+I_2+I_3$ and $I_1+I_2+b+I_3$ are indiscernible over $A$. Let $B$ be any set of parameters. Then we can find $a',b'$ such that $a'\models \lim(I_1/B)$, $b'\models \lim(I_2/B)$ and $(a',b')\equiv_A (a,b)$.
\end{lemme}
\begin{proof}
See Lemma 2.9 of \cite{distal}.
\end{proof}

\begin{lemme}\label{lem_basechangemut}
Let $I_1+I_2$, $J_1+J_2$ be sequences mutually indiscernible over $A$ (all sequences are without endpoints). Let $a$ and $b$ such that $I_1+a+I_2$ and $J_1+J_2$ are mutually indiscernible over $A$ and so are $I_1+I_2$ and $J_1+b+J_2$. Let $B$ be any set of parameters. Then we can find $a'$, $b'$ such that $a'\models \lim(I_1/B)$, $b'\models \lim(J_1/B)$ and $(a',b')\equiv_{A} (a,b)$.
\end{lemme}
\begin{proof}
We may assume that all the sequences considered have same order type as $\mathbb R$. Then we can write $I_1=(a_t :t\in (0,1))$, $I_2=(a_t:t\in (1,3))$, $J_1=(b_t:t\in (0,2))$ and $J_2=(b_t :t\in (2,3))$. For $l\in \{0,1,2\}$, let $K_l$ be the sequence of pairs $\langle (a_t,b_t) :t\in (l,l+1)\rangle$. By mutual indiscernibility, the sequence $K_1+K_2+K_3$ is indiscernible over $A$. Furthermore, we can find $b'_1$ and $a'_2$ such that $K_1+(a,b'_1)+K_2+K_3$ and $K_1+K_2+(a'_2,b)+K_3$ are indiscernible over $A$.

Now apply the previous lemma.
\end{proof}

\begin{Prop}\label{prop_betteramalg}
Let $M\prec^+ N$. Let $p,q,r\in S(\monster)$ be pairwise commuting $M$-invariant types. Let $c\models r|_N$, $a\models p|_{Nc}$, $b\models q|_{Nc}$. Let $B$ be any set containing $Nc$. Then we can find $a',b'$ such that $(a',b')\equiv_{Nc} (a,b)$, $a'\models p|_B$ and $b'\models q|_B$.
\end{Prop}
\begin{proof}
First, we build a sequence $I_1$ such that $I_1+a$ is a Morley sequence of $p$ over $Nc$ and $I_1$ is a Morley sequence of $p$ over $Ncb$. To show that it is possible, fix some small $A\subset N$ containing $M$. Let $I_1^{A}\subset N$ be a Morley sequence of $p$ over $A$. Then by the commutativity assumptions, $I_1^A+a$ is a Morley sequence of $p$ over $Ac$ and $I_1^A$ is a Morley sequence of $p$ over $Acb$. We conclude by compactness.

Next, we construct similarly a sequence $J_1$ such that $J_1+b$ is a Morley sequence of $q$ over $NI_1c$ and $J_1$ is a Morley sequence of $q$ over $NI_1ca$. Finally, we build $I_2$ (resp. $J_2$) a Morley sequence of $p$ (resp. $q$) over everything, including $B$. We take the index set of those sequences to be without endpoints.

Now, using again the commutativity assumptions: $I_1+I_2$ and $J_1+J_2$ are mutually indiscernible over $Nc$ and both $I_1+a+I_2$ and $J_1+b+J_2$ are indiscernible over $Nc$. Let $I_2^*$ (resp. $J_2^*$) be the sequence $I_2$ (resp. $J_2$)  indexed in the opposite order. By construction $\lim(I_2^*/B)=p|_B$ and similarly for $J^*_2$. Thus we can apply the previous lemma to obtain what we want.
\end{proof}

It seems possible that the commutativity assumptions can be somewhat relaxed.

\begin{cor}
Let $(p_i:i<\alpha)$ be a family of global invariant types. Assume that the $p_i$'s are pairwise orthogonal. Then $\bigcup_{i<\alpha} p_i(x_i)$ defines a complete type in the variables $(x_i:i<\alpha)$.
\end{cor}
\begin{proof}
Let $\kappa$ be such that all types considered are invariant over a set of size $<\kappa$.

First assume that $\alpha=3$. So let $p,q,r$ be three invariant types which are pairwise orthogonal. In particular, they pairwise commute. Let $a\models p$, $b\models q$ and $c\models r$. Let $N$ be some $\kappa$-saturated model containing $\monster c$. The types $p,q,r$ have a unique extension to an $M$-invariant type over $N$ which we denote by the same letter. By Proposition \ref{prop_betteramalg}, we can find $a',b'$ such that $a' \models p|N$, $b'\models q|N$ and $(a',b')\equiv_{\monster c} (a,b)$. By orthogonality of $p$ and $q$, we have $(a,b)\models p\otimes q|N$. In particular $(a,b)\models p\otimes q|\monster c$. and $(a,b,c)\models p\otimes q \otimes r$.

The general case follows by induction on $\alpha$.
\end{proof}

\section{Open problems}\label{sec_questions}

We list here a few open problems.

First, let us recall the central conjecture of Section \ref{sec_pq}, which was first stated in \cite{ExtDef2}.

\begin{conj}\label{conj_gen}
Let $T$ be NIP and $M\models T$. Let $\phi(x;d)\in L(\monster)$ a formula, non-forking over $M$. Then there is $\theta(y)\in \tp(d/M)$ such that the partial type $\{ \phi(x;d') : d'\in \theta(\monster)\}$ is consistent.
\end{conj}

In dp-minimal theories, a stronger version seems plausible:

\begin{conj}\label{conj_dpmin}
Assume that $T$ is dp-minimal. If the formula $\phi(x;b)$ does not fork over $M$, then it extends to an $M$-definable type.
\end{conj}

In \cite{SimStar} written with S. Starchenko, we confirm this conjecture assuming in addition that any formula with parameters $b$ extends to a $b$-definable type.

\medskip
Returning to the initial problem which was mentioned in the introduction, we would like to analyse a general invariant type by a finitely satisfiable type and a definable `quotient'. The following is a test question in that direction.

\begin{conj}
Assume that $T$ is NIP. Let $M\pprec N$, $a\in \monster$ such that $\tp(a/N)$ is $M$-invariant, and $\phi(x;y)\in L(M)$. Then there is $b\in \monster$ and a formula $\psi(x;z)$ such that $\tp(b/N)$ is finitely satisfiable in $M$ and $\phi(N;a)=\psi(N;b)$.
\end{conj}

A positive answer to this conjecture, would imply a positive answer to the following:

\begin{conj}
Let $p(x)$ be a global $M$-invariant type and $\Delta$ a finite set of formulas, then there is a finite set $\Delta'$ of formulas such that for any $\phi(x;y)\in \Delta$ and $b,b'\in \monster$, if $\tp_{\Delta'}(b/M)=\tp_{\Delta'}(b'/M)$, then $p\vdash \phi(x;b)\leftrightarrow \phi(x;b')$.
\end{conj}

However, we cannot hope to be able  in general to choose $\Delta'$ independently of $p$ as the following example shows. Let $L=\{<; P_n : n<\omega\}$, where the $P_n$'s are unary predicates. The theory $T$ says that $<$ defines a dense linear order and the predicates $P_n$ name distinct initial segments of it. For each $n$, we have a $\emptyset$-definable type $p_n$ of an element satisfying $P_n$, but greater than all points in $P_n(\monster)$.

Now take $\Delta=\{<\}$. Then the previous conjecture holds for $p_n$ (and any choice of $M$) by taking $\Delta'=\{<,P_n\}$. However, for any $\Delta'$ in which $P_n$ does not appear, we can find two elements $b,b'$ having the same $\Delta'$-type over $M$, such that $b$ satisfies $P_n$ but $b'$ does not. Then we have $p\vdash x>b \wedge x<b'$.

\medskip

Finally, the main open question concerning $F_M$ is the following.

\begin{qu}
Does the map $F_M$ have bounded fibers? In other words, is it the case that for $p$ finitely satisfiable in $M$, the preimage $F_M^{-1}(p)$ of $p$ has cardinality bounded in terms of $|T|$, but independently of $|M|$ (by $2^{|T|}$ for example)? 
\end{qu}

A positive answer would prove for example that if $T$ has medium directionality, then there are at most $|M|^{|T|}$ \emph{invariant} types over any model $M$, and if $T$ has low directionality, then any type $p\in S(M)$ has at most $2^{|T|}$ many invariant extensions.

\section*{Acknowledgments}
Thanks to Gareth Boxall and Charlotte Kestner for pointing out some mistakes in an earlier version. Thanks also to the referee for many useful comments.

Supported by the European Research Council under the European Unions Seventh Framework Programme (FP7/2007-2013) / ERC Grant agreement no. 291111.

\bibliographystyle{unsrt}

\end{document}